\documentclass[11pt]{amsart}
\usepackage{amssymb}
\usepackage{verbatim}
\usepackage{amsmath}
\usepackage{}
\usepackage{amssymb,latexsym}
\usepackage{enumerate}
\newcommand{\op} {\overline{\partial}}
\newcommand{\dbar}{\ensuremath{\overline\partial}}

\newcommand{\C}{\ensuremath{\mathbb{C}}}

\makeatletter
\newcommand{\sumprime}{\if@display\sideset{}{'}\sum%
            \else\sum'\fi}
\makeatother

\begin{document}

\numberwithin{equation}{section}

% define theorem environments
\newtheorem{theorem}{Theorem}[section]
\newtheorem{proposition}[theorem]{Proposition}
\newtheorem{conjecture}[theorem]{Conjecture}
\def\theconjecture{\unskip}
\newtheorem{corollary}[theorem]{Corollary}
\newtheorem{lemma}[theorem]{Lemma}
\newtheorem{observation}[theorem]{Observation}
\newtheorem{definition}{Definition}
\numberwithin{definition}{section} %\def\thedefinition{\unskip}
\newtheorem{remark}{Remark}
\def\theremark{\unskip}
\newtheorem{kl}{Key Lemma}
\def\thekl{\unskip}
\newtheorem{question}{Question}
\def\thequestion{\unskip}
\newtheorem{example}{Example}
\def\theexample{\unskip}
\newtheorem{problem}{Problem}

\thanks{Research supported by Knut and Alice Wallenberg Foundation, and the China Postdoctoral Science Foundation.}

\address{School of Mathematical Sciences, Fudan University, Shanghai, 200433, China}
\address{Current address: Department of Mathematical Sciences, Chalmers University of Technology and
University of Gothenburg. SE-412 96 Gothenburg, Sweden}
\email{wangxu1113@gmail.com}
\email{xuwa@chalmers.se}

\title[Higher direct image]{Curvature of higher direct image sheaves and its application on negative-curvature criterion for the Weil-Petersson metric}
 \author{Xu Wang}
\date{\today}

\begin{abstract} We shall show that $q$-semipositivity of the vector bundle $E$ over a K\"ahler total space $\mathcal X$ implies the Griffiths-semipositivity of the $q$-th direct image of $\mathcal O(K_{\mathcal X/B}\otimes E)$. As an application, we shall give a negative-curvature criterion for the generalized Weil-Petersson metric on the base manifold.

\bigskip

% \noindent{{\sc Mathematics Subject Classification} (2010): 32A25}

\smallskip

\end{abstract}
\maketitle

\tableofcontents

\section{Introduction}

\subsection{Set up} Let $\pi: \mathcal X \to B$ be a proper holomorphic submersion from a complex manifold $\mathcal X$ to a connected complex manifold $B$. We call $B$ the base manifold of the fibration $\pi$. Assume that $B$ is $m$-dimensional and each fibre  $X_t:=\pi^{-1}(t)$ of $\pi$ is $n$-dimensional. Let $E$ be a holomorphic vector bundle over $\mathcal X$. Denote by $E_t$ the restriction of $E$ to $X_t$. Let us denote by $K_{\mathcal X/ B}$ the relative canonical bundle on $\mathcal X$. Recall that
\begin{equation}
K_{\mathcal X/B}=K_{\mathcal X}-\pi^{*}(K_{B}).
\end{equation}

Let us denote by $H^{p,q}(E_t)$ the space of $E_t$-valued $\dbar$-closed $(p,q)$-type Dolbeault cohomology classes over $X_t$. Let us consider
\begin{equation}\label{eq:h}
\mathcal H^{p,q}:=\{H^{p,q}(E_t)\}_{t\in B}.
\end{equation}
It is known that $\mathcal H^{p,q}$ has a natural holomorphic vector bundle structure if 

\medskip

\textbf{A:}  \emph{The dimension of $H^{p,q}(E_t)$ does not depend on $t\in B$.}

\medskip

The usual way to prove this fact is to use the base change theorem. In fact, by the base change theorem, $\mathbf A$ implies that the $q$-th direct image sheaf
\begin{equation}
\underline{E}^{p,q}:=R^{q} \pi_*\mathcal O(\wedge^{p} T^*_{\mathcal X/B} \otimes E).
\end{equation}
is locally free and the fibres of the holomorphic vector bundle associated to $\underline{E}^{p,q}$ are 
\begin{equation}
H^{q}(X_t, \mathcal O(\wedge^{p} T^*_{X_t} \otimes E_t )),  \ \ t\in B.
\end{equation}
By Dolbeault's theorem,
\begin{equation}
H^{q}(X_t, \mathcal O(\wedge^{p} T^*_{X_t} \otimes E_t )) \simeq  H^{p,q}(E_t),
\end{equation}
thus there is a holomorphic vector bundle structure on $\mathcal H^{p,q}$.

In this paper, we shall use the Newlander-Nirenberg theorem \cite{NN57} and a theorem of Kodaira-Spencer (see Page 347 in \cite{KSp}) to construct a holomorphic vector bundle structure on $\mathcal H^{p,q}$ directly (see Theorem \ref{th:holomorphic} below).

Assume that $\mathcal X$ possesses a K\"ahler form  $\omega$, put
\begin{equation}\label{eq:omegat}
\omega^t:=\omega|_{X_t}>0, \ \text{on}\ X_t.
\end{equation}
Let $h$ be a Herimitian metric on $E$. By the Hodge theory, every cohomology class in $H^{p,q}(E_t)$ has a unique harmonic representative in 
\begin{equation}\label{eq:dbar-harmonic}
\mathcal H^{p,q}_t:=\ker \dbar^t \cap \ker (\dbar^t)^*,
\end{equation}
where $\dbar^t:=\dbar|_{X_t}$ and $(\dbar^t)^*$ is the adjoint of $\dbar^t$ with respect to $\omega^t$ and $h$. In this paper, we shall identify a cohomology class in $H^{p,q}(E_t)$ with its harmonic representative in $\mathcal H^{p,q}_t$. Thus we have 
\begin{equation}\label{eq:H-harmonic}
\mathcal H^{p,q}=\{\mathcal H^{p,q}_t\}_{t\in B}.
\end{equation}
The natural $L^2$-inner product on each harmonic space $\mathcal H^{p,q}_t$ defines a Hermitian metric on $\mathcal H^{p,q}$. Let us denote by $D$ the associated Chern connection on $\mathcal H^{p,q}$. Put
\begin{equation}\label{eq:D}
D=\sum dt^j\otimes D_{t^j} +d\bar t^j \otimes \dbar_{t^j}.
\end{equation}
Then
\begin{equation}\label{eq:D2}
D^2=\sum dt^j\wedge d\bar t^k \otimes \Theta_{j\bar k}, \  
\Theta_{j\bar k} :=D_{t^j}\dbar_{t^k}-\dbar_{t^k}D_{t^j} .
\end{equation}
Denote by $(\cdot,\cdot)$ and $||\cdot||$ the assocaiated inner product and norm on the fibre, $\mathcal H^{p,q}_t$, of $\mathcal H^{p.q}$ respectively. By definition, $\mathcal H^{p,q}$ is semi-positive in the sense of Griffiths if and only if
\begin{equation*}
\sum (\Theta_{j\bar k} u, u)\xi^j\bar \xi^k \geq 0, \ \forall \ u\in\mathcal H^{p,q}_t, \ \xi\in\mathbb C^m.
\end{equation*}
Moreover,  $\mathcal H^{p,q}$ is semi-positive in the sense of Nakano if and only if 
\begin{equation*}
\sum (\Theta_{j\bar k} u_j, u_k)\geq 0, \ \forall \ u_j\in\mathcal H^{p,q}_t, \ 1\leq  j\leq m.
\end{equation*}
In this paper, we shall give a curvature formula for $\mathcal H^{p,q}$ based on the formulas in \cite{Bern06, Bern09, Bern11} for $\mathcal H^{n,0}$, see \cite{Maitani84} and \cite{MY04} for earlier results and  \cite{Tsuji05}, \cite{Sch12}, \cite{LiuYang13}, \cite{GS15}, \cite{MT07} and \cite{MT08} for other generalizations and related results.  

\subsection{Main result} 

Denote by $\Theta(E,h)$ (resp. $\Theta(E_t,h)$) the curvature operator of the Chern connection
\begin{equation}
d^E:= \dbar +\partial^E, \ (\text {resp. } d^{E_t}:=\dbar^t +\partial^{E_t} ) ,
\end{equation} 
on $(E,h)$ (resp. $(E_t,h)$) respectively.  Put 
\begin{equation}
\omega_q:=\frac{\omega^q}{q!}, \ \ \omega^t_q:=\frac{(\omega^t)^q}{q!}.
\end{equation}
We shall use the following definition:

\begin{definition}\label{de:q-positive} $\Theta(E,h)$ is said to be $q$-semipositive on $\mathcal X$ with respect to $\omega$ if 
\begin{equation}
\omega_q\wedge  c_q\{i\Theta(E,h)u, u\} \geq 0,  \ \text{on}\ \mathcal X,
\end{equation}
for every $E$-valued $(n+m-q-1,0)$-form $u$ in $\mathcal X$. where $c_q=i^{(n+m-q-1)^2}$ is chosen such that
\begin{equation}
\omega_{q+1} \wedge c_q\{u,u\} \geq 0,
\end{equation} 
as a semi-positive volume form on the total space $\mathcal X$.
\end{definition}

Our main theorem is the following:

\begin{theorem}\label{th:main} Assume that the total space is K\"ahler and the dimension of $H^{n,q}(E_t)$ does not depend on $t\in B$. Assume further that $\Theta(E,h)$ is $q$-semipositive with respect to a K\"ahler form on the total space. Then  $\mathcal H^{n,q}$ is Griffiths-semipositive.
\end{theorem}

% \textbf{Remark 1}: As a direct consequence of the Nadel vanishing theorem (see \cite{Nadel89} and the proof of Theorem 5.11 in \cite{Demailly10}), we know that \eqref{eq:main1} and $\mathbf A$ together imply $A2$. Moreover, from Berndtsson's $L^2$-extension theorem for $\dbar$-closed form (see \cite{Bern12} and Theorem 13.6 in \cite{Demailly10} for a more general version), we also know that \eqref{eq:main1} and $\mathbf B$ together imply $A2$. 

\textbf{Remark}: The following fact is also true: 

\medskip

\emph{With the assumptions in the above theorem, assume further that 
\begin{equation}\label{eq:main2}
\mathcal H^{n,q}_t= \ker \partial^{E_t} \cap \ker (\partial^{E_t})^*, \ \forall  \ t\in B,
\end{equation}
then $\mathcal H^{n,q}$ is Nakano-semipositive.}

\medskip 

If $q=0$ then \eqref{eq:main2} is always true. If $q\geq 1$, by Siu's $\partial\dbar$-Bochner formula (see \cite{Siu82} or \cite{Bern02}), we know that  \eqref{eq:main2} is true in case
\begin{equation}
 i\Theta(E_t,h)\wedge\omega^t_{q-1}\equiv 0, 
\end{equation} 
on $X_t$ for every $t\in B$.

\subsection{Applications} We shall use our main theorem to study the curvature properties of the base manifold $B$. Let us denote by $\kappa$ the Kodaira-Spencer mapping
\begin{equation}
\kappa :v \mapsto \kappa(v) \in H^{0,1}(T_{X_t})\simeq H^{n,n-1}(T^*_{X_t})^*, \ \forall  \ v\in T_tB, \ t\in B.
\end{equation} 
We shall introduce the following definition:

\begin{definition} We call the pull back pseudo-metric on $T_B$ defined by
\begin{equation}
||v||_{WP}:=||\kappa(v)||_{H^{n,n-1}(T^*_{X_t})^*}, \ \forall  \ v\in T_tB
\end{equation}
the generalized Weil-Petersson metric on $B$.
\end{definition}

Assume that the dimension of $H^{0,1}(T_{X_t})$ does not depend on $t$ in $B$. Then we know that the dimension of the dual space, $H^{n,n-1}(T^*_{X_t})$, of $H^{0,1}(T_{X_t})$ does not depend on $t$ in $B$. Moreover, we shall prove that (see Proposition \ref{pr:ks-dual}) 
\begin{equation}
\kappa: v\mapsto \kappa(v)\in H^{n,n-1}(T^*_{X_t})^*,  \ \forall  \ v\in T_tB,
\end{equation} 
defines a holomorphic bundle map from $T_B$  to the \textbf{dual} bundle of
\begin{equation}
\mathcal H^{n,n-1}:=\{H^{n,n-1}(T^*_{X_t})\}_{t\in B}, 
\end{equation}
Assume that the total space $\mathcal X$ possesses a K\"ahler form $\omega$. We shall introduce the following definition:

\begin{definition} The relative cotangent bundle $T^*_{\mathcal X/B}$ is said to be $(n-1)$-semipositive with respect to $\omega$ if there is a smooth metric, say $h$, on the relative cotangent bundle such that $\Theta(T^*_{\mathcal X/B}, h)$ is $(n-1)$-semipositive with respect to $\omega$.
\end{definition}

By Theorem \ref{th:main}, if the relative cotangent bundle is $(n-1)$-semipositive with respect to $\omega$ then $\mathcal H^{n,n-1}$ is Griffiths-semipositive. Assume further that $\kappa$ is injective. Then $T_B$ is Griffiths-seminegative with respect to the generalized Weil-Petersson metric. Inspired by \cite{Raufi15} and \cite{BernPaun08}, we shall introduce the following definition:

\begin{definition} $||\cdot||_{WP}$ defines a Griffiths-seminegative singular metric  if and only if for every holomorphic vector field $v: t \mapsto v^t$ on the base, $\log ||v||_{WP}$ is plurisubharmonic or equal to $-\infty$ identically, where $||v||_{WP}$ denotes the upper semicontinuous regularization of $t\mapsto ||v^t||_{WP}$.
\end{definition}

We shall prove that:

\begin{theorem}\label{th:wp-general} Let $\pi$ be a proper holomorphic submersion from a K\"ahler manifold $(\mathcal X, \omega)$ to a complex manifold $B$. Assume that the relative cotangent bundle is $(n-1)$-semipositive with respect to $\omega$. Then the associated generalized Weil-Petersson metric defines a Griffiths-seminegative singular metric on $T_B$. 
\end{theorem}

\textbf{Remark A}: If the canonical line bundle of each fibre is positive then by Aubin-Yau's theorem (see \cite{Aubin78} and \cite{Yau78}), each fibre possesses a unique K\"ahler-Einstein metric, which defines a smooth Hermitian metric, say $h$, on the relative cotangent bundle. Then we know that the cotangent bundle of \emph{each fibre} is $(n-1)$-semipositive. Moreover, if $n=1$ then it is well known that relative cotangent bundle is $0$-semipositive. But for a general canonically polarized family, we don't know whether the relative cotangent bundle is $(n-1)$-semipositive or not, for related results, say \cite{Sch12}. In \cite{BPW16}, we shall introduce another way to study the curvature properties of the base manifold of  a canonically polarized family.

\medskip

\textbf{Remark B}: Given a K\"ahler total space  $(\mathcal X, \omega)$, if the canonical line bundle of each fibre is trivial then by Yau's theorem \cite{Yau78}, we know that there is a smooth function, say $\phi$, on $\mathcal X$ such that $\omega+i\partial\dbar\phi$ is Ricci-flat on each fibre. Let us denote by $h$ the smooth Hermitian metric on the relative cotangent bundle defined by $\omega+i\partial\dbar\phi$. Then we know that the cotangent bundle of \emph{each fibre} is $(n-1)$-semipositive. But we don't know whether the relative cotangent bundle is $(n-1)$-semipositive or not, except for some special case, e.g. deformation of torus or other families with flat relative cotangent bundle. In \cite{Wang16-1}, we shall introduce another metric to study the base manifold of a Calabi-Yau family.

\subsection{List of notations} \ \

\medskip

\textbf{Basic notions:}

1. $\pi: \mathcal X \to B$ is a proper holomorphic submersion, $E$: holomorphic vector bundle on $\mathcal X$;

2. $X_t:=\pi^{-1}(t)$ is the fibre at $t$, $E_t:=E|_{X_t}$;

3. $d^E:=\dbar+\partial^E$ is the Chern connection on $E$, $d^{E_t}:=\dbar^t +\partial^{E_t} $ is its restriction;

4. $\Theta(E,h):=(d^E)^2$ is Chern curvature of $E$, $\Theta(E_t,h):=(d^{E_t})^2$;

5. $H^{p,q}(E_t)$:  Dolbeault cohomology group on $X_t$;

6. $\mathcal H^{p,q}_t := \ker \dbar^t \cap \ker (\dbar^t)^* \simeq H^{p,q}(E_t)$ is the harmonic space;

7. $\mathcal H^{p,q}:=\{\mathcal H^{p,q}_t\}_{t\in B} \simeq \{H^{p,q}(E_t)\}_{t\in B} $;

8. $D=\sum dt^j\otimes D_{t^j} +d\bar t^j \otimes \dbar_{t^j}$ is the Chern connection on $\mathcal H^{p,q}$;

9. $\Theta_{j\bar k}:=D_{t^j}\dbar_{t^k}-\dbar_{t^k}D_{t^j}$ is the curvature operator on $\mathcal H^{p,q}$.

\medskip

\textbf{Other notations:}

1. $i_t$: the inclusion mapping $X_t\hookrightarrow \mathcal X$;

2. $t$: local holomorphic coordinate system on $B$, $t^j$: components of $t$;

3. $\delta_{V}:=V~\lrcorner ~$ means contraction of a form with a vector field $V$;

4. $V_j$: smooth $(1,0)$-vector field on $\mathcal X$ such that $\pi_* V_j =\partial/\partial t^j$, $L_{V_j}$: usual Lie-derivative;

5. $L_j:=d^E\delta_{V_j}+\delta_{V_j} d^E=[d^E, \delta_{V_j}], \ \ L_{\bar j} :=d^E \delta_{\overline{V_j}} +\delta_{\overline{V_j}} d^E =[d^E, \delta_{\overline{V_j}}]$;

6. $u: t\mapsto u^t\in \mathcal H^{p,q}_t$ is a section of $\mathcal H^{p,q}$;

7. $\Gamma(\mathcal H^{p,q})$: space of smooth sections of $\mathcal H^{p,q}$;

8. a smooth $E$-valued $(p,q)$-form $\mathbf u$ on $\mathcal X$: a representative of $u\in \Gamma(\mathcal H^{p,q})$;

9: $*$: Hodge-Poincar\'e-de Rham star operator;

10: $\mathbf u^*$: dual-representative of $u$ such that $ \mathbf u^*|_{X_t}=*u^t$ for every $t\in B$;

11: $\mathbb H^t$: orthogonal projection to $\mathcal H^{p,q}_t$.

12: $\kappa : \partial/\partial t^j \mapsto \kappa(\partial/\partial t^j)$ is the Kodaira-Spencer mapping.

\section{Motivation: the product case} 

\subsection{Griffiths-positivity of the bundle of harmonic forms} We shall give an example to show the ideas behind the proof of Theorem \ref{th:main}.

Let us consider the following product case: Let $L$ be a holomorphic line bundle over a compact K\"ahler manifold $(X,\omega)$. Let $h$ be a fixed smooth metric on $L$ and let $\phi$ be a smooth function on
\begin{equation}
\mathcal X:=X\times \mathbb B.
\end{equation} 
Put
\begin{equation}
h^t=he^{-\phi^t}, \ \ \phi^t:=\phi|_{X\times\{t\}}.
\end{equation}
Thus $\{h^t\}_{t\in B}$ defines a smooth metric, say $\tilde h$, on $L\times \mathbb B$. We shall consider
\begin{equation}
\mathcal H^{n,q}:=H^{n,q}(L)\times \mathbb B=\{\mathcal H^{n,q}_t\}_{t\in\mathbb B},
\end{equation} 
where $\mathbb B$ is the unit ball in $\C^m$ and each $H^{n,q}_t$ is the harmonic space with respect to $\omega$ and $h^t$. We know that $\mathcal H^{n,q}$ is a trivial vector bundle with non-trivial metric. 

For every $u\in H^{n,q}(L)$ there is an associated holomorphic section
\begin{equation}
 t \mapsto u^t:=\mathbb H^t(\mathbf u),
\end{equation}
where each $\mathbf u$ is a fixed $\dbar_X$-closed representative of $u$ and $\mathbb H^t$ is the orthogonal projection to the harmonic space $\mathcal H^{n,q}_t$. Now we know that
\begin{equation}
(u,v): t\mapsto (u^t, v^t),  \ \ u,v\in H^{n,q}(L)
\end{equation}
is a smooth function on $\mathbb B$. Moreover, we have
\begin{equation}
(u,v)_{j\bar k}= (D_{t^j}u^t, D_{t^k}v^t)-(\Theta_{j\bar k}u^t,v^t).
\end{equation}
By definition, we have
\begin{equation}
(u^t, v^t)= \int_X \{u^t, *v^t\} e^{-\phi^t},
\end{equation}
where $\{\cdot, \cdot\}$ is the pairing associated to $h$. Thus we have
\begin{equation}
(u,v)_j=(L_j u^t, v^t)+ \int_X \{u^t, L_{\bar j}*v^t\} e^{-\phi^t},
\end{equation}
where
\begin{equation}
L_j:=\partial/\partial t^j-\phi_j\cdot,  \ \ L_{\bar j}:=\partial/\partial \bar t^j.
\end{equation}
Since $(u^t)^* \wedge \omega_q=i^{(n-q)^2}(-1)^{n-q} u^t$, we have
\begin{equation}\label{eq:L*}
L_j*=*L_j, \ \ L_{\bar j}*=*L_{\bar j}, \ \forall \ 1\leq j\leq m.
\end{equation}
Now we have
\begin{equation}
(u,v)_j=(L_j u^t, v^t)+ (u^t, L_{\bar j}v^t).
\end{equation}
Let $v$ be a fixed (does not depend on $t$) $\dbar_X$-closed representative of $v$. Thus each $v^t-\mathbf v$ is $\dbar_X$-exact, which implies that
\begin{equation}
L_{\bar j} v^t=L_{\bar j} (v^t-\mathbf v),
\end{equation}
is $\dbar_X$-exact. Then we have
\begin{equation}
(u^t,  L_{\bar j} v^t) \equiv 0.
\end{equation}
Hence we get that
\begin{equation}
D_{t^j}u^t =\mathbb H^t(L_j u^t),
\end{equation}
and
\begin{equation}
(u,v)_{j\bar k}=(L_ju^t, v^t)_{\bar k}=(L_{\bar k}L_j u^t, v^t)+(L_j u^t, L_k v^t).
\end{equation}
Notice that
\begin{equation}
L_{\bar k}L_j u^t =L_{j}L_{\bar k}u^t -\phi_{j\bar k} u^t.
\end{equation}
Then we have
\begin{equation}
(L_{\bar k}L_j u^t, v^t)= -(\phi_{j\bar k} u^t,  v^t) + (L_{\bar k}u^t, v^t)_j -(L_{\bar k}u^t, L_{\bar j}v^t).
\end{equation}
Since $(L_{\bar k}u^t, v^t) \equiv 0$, thus we have
\begin{equation}
(u,v)_{j\bar k}=-(\phi_{j\bar k} u^t,  v^t)  -(L_{\bar k}u^t, L_{\bar j}v^t) +(L_j u^t, L_k v^t).
\end{equation}
The following lemma is a crucial step (see Lemma  \ref{le:abc-later} for the general case).

\begin{lemma} Put $a^t_j=D_{t^j} u^t-L_j u^t$, then each $a_j^t$ is the $L^2$-minimal solution of
\begin{equation}
\dbar_X a_j^t=-\dbar_X L_j u^t=\dbar_X\phi_j\wedge u^t.
\end{equation}
\end{lemma}

\begin{proof} It suffices to show that $\dbar_X^*(a^t)\equiv 0$. Since 
$$ 
\dbar_X^*=-*(\partial_X-\partial_X\phi^t)*, \ 
\dbar_X^*(a^t)=-\dbar_X^*(L_j u^t),
$$ 
it is enough to show $*L_j u^t$ is $(\partial_X-\partial_X\phi^t)$-closed. By \eqref{eq:L*}, we have
\begin{equation}
*L_j u^t =L_j *u^t.
\end{equation}
Thus 
\begin{equation}
(\partial_X-\partial_X\phi^t)  *L_j u^t =(\partial_X-\partial_X\phi^t)  L_j * u^t.
\end{equation}
Since each $u^t$ is harmonic, thus 
\begin{equation}
(\partial_X-\partial_X\phi^t)  * u^t \equiv 0.
\end{equation}
Thus it is enough to show that
\begin{equation}
[\partial_X-\partial_X\phi^t ,  L_j]\equiv 0,
\end{equation}
which follows by direct computation.
\end{proof}

Fix $u_1,\cdots, u_m\in H^{n,q}(L)$, then we have
\begin{equation}\label{eq:theta-product}
\sum (\Theta_{j\bar k} u_j^t, u_k^t)=-||a^t||^2+\sum (\phi_{j\bar k} u_j^t, u_k^t)+\sum (L_{\bar k}u_j^t, L_{\bar j}u_k^t) ,
\end{equation}
where
\begin{equation}
a^t:= \sum a_j^t, \ \ \dbar_X a^t=\sum \dbar_X\phi_j\wedge u_j^t:=c^t
\end{equation}
By the classical Bochner-Kodaira-Nakano formula, if 
\begin{equation}\label{eq:motivation-positivity}
i^{(n-q-1)^2}\omega_q\wedge \{i\Theta(L,h^t)u,u\} >0, \ \text{on}\ X,
\end{equation} 
for every $(n-q-1,0)$-form $u$ that has no zero point in $X$ then we have
\begin{equation}
||a^t||^2 \leq \left([i\Theta(L,h^t), \Lambda_\omega]^{-1} c^t , c^t\right)=\sum  \left(T_{j\bar k} u_j^t , u_k^t\right),
\end{equation}
where $\Lambda_\omega$ is the adjoint of $\omega \wedge$ and
\begin{equation}
T_{j\bar k}:= (\dbar_X\phi_k\wedge\cdot)^*[i\Theta(L,h^t), \Lambda_\omega]^{-1} (\dbar_X\phi_j\wedge\cdot)
\end{equation}
Now we have
\begin{equation}\label{eq:theta-product-1}
\sum (\Theta_{j\bar k} u_j^t, u_k^t)= R+\sum\left( (\phi_{j\bar k}-T_{j\bar k}) u_j^t, u_k^t\right)+\sum (L_{\bar k}u_j^t, L_{\bar j}u_k^t),
\end{equation}
where
\begin{equation}
R:= \left([i\Theta(L,h^t), \Lambda_\omega]^{-1} c^t , c^t\right)- ||a^t||^2\geq 0.
\end{equation}
We shall use the following lemma (see Lemma \ref{le: ABC} for the general case):

\begin{lemma} Assume that \eqref{eq:motivation-positivity} is true for every $t\in \mathbb B$. Then we have
\begin{equation}
\sum\left( (\phi_{j\bar k}-T_{j\bar k}) u_j^t, u_k^t\right) \wedge i^{m^2}dt\wedge \overline{dt}=\pi_*\left( i^{(m+n-1-q)^2} \omega_q\wedge\{i\Theta( L\times \mathbb B,\tilde h)\mathbf u^*, \mathbf u^*\}\right),
\end{equation}
where $\mathbf u^*:= \mathbf u_j^* \wedge (\partial/\partial t^j \lrcorner ~dt)$ and each $\mathbf u_j^*$ satisfies that
\begin{equation}
i_t^* \left( \partial/\partial t^j  \lrcorner ~ (\omega_q \wedge \Theta( L\times \mathbb B,\tilde h) \mathbf u_j^*) \right)\equiv 0, \ \ i_t^* \mathbf u_j^* =*u_j^t,
\end{equation}
on $X$ for every $t\in\mathbb B$.
\end{lemma}

Thus we get the following result (see Theorem \ref{th:main} for the general case):

\begin{theorem}\label{th:2.3} Assume that 
\begin{equation}
\omega_q \wedge i\Theta( L\times \mathbb B,\tilde h) \geq 0, \ \ \text{on} \ \mathcal X,
\end{equation}
and
\begin{equation}
\omega_q \wedge  i\Theta( L,h^t) >0, \ \ \text{on} \ X, \ \text{for all} \ t\in \mathbb B.
\end{equation}
Then we have
\begin{equation}
\sum (\Theta_{j\bar k} u_j^t, u_k^t) \geq \sum (L_{\bar k}u_j^t, L_{\bar j}u_k^t).
\end{equation}
In particular, $\mathcal H^{n,q}$ is Griffiths-semipositive.
\end{theorem}

We shall show in the next section that $\mathcal H^{n,q}$ can be seen as a holomorphic quotient bundle of a Nakano-semipositive bundle. But in general,  a holomorphic quotient  bundle of a Nakano-semipositive bundle is not Nakano-semipositive (see Page 340 in \cite{Demailly12} for a counterexample).

\subsection{Nakano-positivity of the bundle of $\dbar$-closed forms}

Let us denote by $\ker \dbar$ (resp. ${\rm Im} ~ \dbar$) the space of smooth $\dbar$-closed (resp. $\dbar$-exact) $L$-valued $(n,q)$-forms on $X$ respectively. Then we have the following trivial bundles:
\begin{equation}
\mathcal K:= \ker \dbar \times \mathbb B, \ \  \mathcal I:={\rm Im} ~ \dbar \times \mathbb B.
\end{equation}
But in general the metrics on $\mathcal K$ and $\mathcal I$ defined by $\{h^t\}_{t\in \mathbb B}$ are not trivial. By definition, we know that $\mathcal H^{n,q}$ is just the quotient bundle $\mathcal K/\mathcal I$. And the metric on $\mathcal H^{n,q}$ is just the quotient metric (see \cite{Wang16-0} for more results).

For every $u,v \in \ker\dbar$, we shall write
\begin{equation}
(u,v): t\mapsto \int_X \{ u, *v\}e^{-\phi^t}.
\end{equation}
Let us denote by $\Theta^{\mathcal K}_{j\bar k}$ the curvature operators on $\mathcal K$.  Fix $u_1,\cdots, u_m  \in \ker\dbar$. Since now
\begin{equation}
L_{\bar j} u_k\equiv 0, \ \ \forall \ 1\leq j,k \leq m,
\end{equation}
we know that the following theorem is true.

\begin{theorem} With the assumptions in Theorem \ref{th:2.3}, then we have
\begin{equation}
\sum (\Theta^{\mathcal K}_{j\bar k} u_j, u_k) = R+\sum\left( (\phi_{j\bar k}-T_{j\bar k}) u_j, u_k\right) \geq 0.
\end{equation}
In particular, $\mathcal K$ is Nakano-semipositive.
\end{theorem}

\textbf{Remark:} One may also study the positivity properties of the bundle of $\dbar$-closed forms for non-trivial fibrations (see \cite{Wang16}). 

\section{Curvature formula}

We shall give a curvature formula for $\mathcal H^{p,q}$ in this section.

\subsection{Holomorphic vector bundle structure on $\mathcal H^{p,q}$} By a theorem of Kodaira-Spencer (see Page 349 in \cite{KSp}), we know that $\mathcal H^{p,q}$ has a smooth complex vector bundle structure if $\mathbf A$ is true. More precisely, $\mathbf A$ implies that for every $t_0\in B$ and every $u^{t_0}\in \mathcal H_{t_0}^{p,q}$ (see \eqref{eq:dbar-harmonic} for the definition of the $\dbar$-harmonic space $\mathcal H_{t_0}^{p,q}$), there is a smooth $E$-valued $(p,q)$-form, say $\mathbf u$, on $\mathcal X$ such that
\begin{equation}
\mathbf u|_{X_{t_0}}=u^{t_0},
\end{equation} 
and 
\begin{equation}
\mathbf u|_{X_{t}}\in \mathcal H_{t}^{p,q},
\end{equation} 
for every $t\in B$. Then the smooth vector bundle structure $\mathcal H^{p,q}$ can be defined as follows:

\begin{definition}\label{de:smooth} We call $u: t\mapsto u^t\in\mathcal H^{p,q}_t$ a smooth section of $\mathcal H^{p,q}$ if there exists a smooth $E$-valued $(p,q)$-form, say $\mathbf u$, on $\mathcal X$ such that $\mathbf u|_{X_{t}}=u^{t}$, $\forall \ t\in B$. And we call $\mathbf u$ a representative of $u$. We shall denote by $\Gamma(\mathcal H^{p,q})$ the space of smooth sections of $\mathcal H^{p,q}$.
\end{definition}

By using the Newlander-Nirenberg theorem, we shall prove that:

\begin{theorem}\label{th:holomorphic} Assume that $\mathcal H^{p,q}$ satisfies $\mathbf A$. Then $D^{0,1}:=\sum d\bar t^j \otimes \dbar_{t^j}$ defines a holomorphic vector bundle structure on $\mathcal H^{p,q}$, where each $\dbar_{t^j}$ is defined by
\begin{equation}
\dbar_{t^j} u:  t \to \mathbb H^t\left(i_t^* [\dbar, \delta_{\overline{V_j}}]  \mathbf{u}\right), \ \ [\dbar, \delta_{\overline{V_j}}]:=\dbar\delta_{\overline{V_j}}+ \delta_{\overline{V_j}}\dbar.
\end{equation}
Here $\textbf{u}$ is an arbitrary representative of $u\in\Gamma(\mathcal H^{p,q})$, $\mathbb H^t$ denotes the orthogonal projection to $\mathcal H^{p,q}_t$ and $V_j$ is an arbitrary smooth $(1,0)$-vector field on $\mathcal X$ such that $\pi_*V_j=\partial/\partial t^j$.
\end{theorem}

\begin{proof} First, let us show that $D^{0,1}$ is well defined. Since each $u^t\in \mathcal H^{p,q}_t$ is harmonic, we know that $i_t^*\dbar \mathbf u\equiv 0$, thus the definition of $\dbar_{t^j} u$ does not depend on the choice of $V_j$. Thus it suffices to check that  $\dbar_{t^j} u$ does not depend on the choice of $\mathbf u$. Let $\mathbf u'$ be another representative of $u$ then we have
\begin{equation*}
\mathbf u- \mathbf{u}'=\sum dt^j\wedge a_j +\sum d\bar t^k \wedge b_k.
\end{equation*}
Thus
\begin{equation}
i_t^*\delta_{\overline{V_j}}\dbar (\mathbf u- \mathbf{u}') =-i_t^* \dbar b_j,
\end{equation}
which implies that
\begin{equation}
\mathbb H^t\left(i_t^* [\dbar, \delta_{\overline{V_j}}]  (\mathbf{u}-\mathbf{u'})\right)=0.
\end{equation}
Thus $D^{0,1}$ is well defined. It is easy to check that
\begin{equation}
\dbar_{t^j} (fu)=f \dbar_{t^j} u +f_{\bar j} u,
\end{equation}
where $f$  is an arbitrary smooth function on $B$. By Newlander-Nirenberg's theorem, it suffices to show that $D^{0,1}$ is integrable, i.e. $(D^{0,1})^2 =0$. By definition, it is sufficient to show
\begin{equation}
\dbar_{t^j} \dbar_{t^k} u=\dbar_{t^k}\dbar_{t^j}u,
\end{equation}
for every $u\in\Gamma(\mathcal H^{p,q})$ and every $1\leq j,k \leq m$. 

Notice that 
\begin{equation}
\mathbf u_j:= [\dbar, \delta_{\overline{V_j}}]  \mathbf{u},
\end{equation}
satisfies (see Lemma \ref{le:re-u-1}, since $ i_t^*\dbar \mathbf u\equiv 0$)
\begin{equation}
\dbar^t i_t^*\mathbf u_j =  i_t^* \dbar \mathbf u_j =  i_t^* [\dbar, \delta_{\overline{V_j}}]  \dbar \mathbf u=0.
\end{equation}
Thus each $i_t^*\mathbf u_j$ has the following orthogonal decomposition
\begin{equation}
i_t^*\mathbf u_j=(\dbar_{t^j}u)(t)+\dbar^t(\dbar^t)^*G^t( i_t^*\mathbf u_j),
\end{equation}
where $G^t$ is the Green operator on $X_t$. Using Kodaira-Spencer's theorem again, we know that $G^t( i_t^*\mathbf u_j )$ depends smoothly on $t$. Thus $(\dbar^t)^*G^t( i_t^*\mathbf u_j )$ depends smoothly on $t$. For each $j$, let us choose a smooth form $\mathbf v_j$ on $\mathcal X$  such that 
\begin{equation*}
i_t^*\mathbf v_j=(\dbar^t)^*G^t( i_t^*\mathbf u_j ).
\end{equation*} 
By definition, we know that each
\begin{equation}
\mathbf u_j-\dbar \mathbf v_j,
\end{equation}
is a representative of $\dbar_{t^j}u$.  Thus we have
\begin{eqnarray}
\left(\dbar_{t^j} \dbar_{t^k} -\dbar_{t^k}\dbar_{t^j}\right)u &=& \mathbb H^t\left(i_t^* [\dbar, \delta_{\overline{V_j}}] (\mathbf u_k-\dbar \mathbf v_k)-  i_t^* [\dbar, \delta_{\overline{V_k}}] (\mathbf u_j-\dbar \mathbf v_j) \right)\\
&=& \mathbb H^t i_t^* \left( [\dbar, \delta_{\overline{V_j}}]   [\dbar, \delta_{\overline{V_k}}] -[\dbar, \delta_{\overline{V_k}}]   [\dbar, \delta_{\overline{V_j}}]  \right) \mathbf u.
\end{eqnarray}
Notice that
\begin{equation}
[L_{\overline{V_j}}, L_{\overline{V_k}}] =L_{[\overline{V_j}, \overline{V_k}]},
\end{equation}
implies that
\begin{equation}
[\dbar, \delta_{\overline{V_j}}]   [\dbar, \delta_{\overline{V_k}}] -[\dbar, \delta_{\overline{V_k}}]   [\dbar, \delta_{\overline{V_j}}]  =[\dbar,  \delta_{[\overline{V_j}, \overline{V_k}]} ]
\end{equation}
Since each $u^t$ is harmonic (thus $\dbar^t$-closed), we have
\begin{equation}
\left(\dbar_{t^j} \dbar_{t^k} -\dbar_{t^k}\dbar_{t^j}\right)u  = \mathbb H^t i_t^* [\dbar,  \delta_{[\overline{V_j}, \overline{V_k}]} ] \mathbf u= \mathbb H^t i_t^* \delta_{[\overline{V_j}, \overline{V_k}]} \dbar \mathbf u=   \mathbb H^t \delta_{[\overline{V_j}, \overline{V_k}]|_{X_t}} \dbar^t u^t=0. 
\end{equation}
The proof is complete.
\end{proof}

\subsection{Chern connection on $\mathcal H^{p,q}$} In this subsection, we shall define the Chern connection on $\mathcal H^{p,q}$. 

Assume that $\mathcal X$ possesses a K\"ahler form, say $\omega$. Then each fibre $X_t$ possesses a K\"ahler form $\omega^t:=\omega|_{X_t}$.  Recall that a smooth $k$-form $\alpha$ on $X_t$ is said to be primitive with respect to $\omega^t$ if $k\leq n$ and $\omega^t_{n-k+1} \wedge \alpha =0$ on $X_t$. Let $u: t\mapsto u^t\in\mathcal H^{p,q}_t$ be a smooth section of $\mathcal H^{p,q}$.
By the Lefschetz decomposition theorem, each $u^t$ has a unique decomposition of the form
\begin{equation}
u^t:= \sum_r \omega^t_{r} \wedge u^t_r,
\end{equation}
where each $u^t_r$ is a smooth $E_t$-valued primitive $(p-r, q-r)$-form. Let us denote by $*$ the Hodge-Poincar\'e-de Rham star operator with respect to $\omega^t$. Then we have
\begin{equation}\label{eq:star-lefschetz}
*u^t:= \sum_r C_r ~\omega^t_{n+r-p-q} \wedge u^t_r,  \ \ \ C_r =i^{(p+q-2r)^2}(-1)^{p-r}.
\end{equation}
Since $u^t$ depends smoothly on $t$, we know that each $u^t_r$ also depends smoothly on $t$. Thus for each $r$, there exists a smooth $E$-valued $(p-r, q-r)$-form, say $\mathbf {u_r}$ on $\mathcal X$ such that 
\begin{equation}
\mathbf u_r |_{X_t}=u^t_r.
\end{equation}
By Definition \ref{de:smooth}, we know that
\begin{equation}
\sum_r \omega_{r} \wedge \mathbf u_r
\end{equation}
is a representative of $u$. We shall use the following definition:

\begin{definition}\label{de:dual-repre}  We call a smooth $E$-valued $(n-q,n-p)$-form $\mathbf u^*$ on $\mathcal X$ a dual-representative of $u \in\Gamma(\mathcal H^{p,q})$ if 
\begin{equation}
\mathbf u^* =  \sum_r C_r ~\omega_{n+r-p-q} \wedge\mathbf  u_r.
\end{equation}
\end{definition}

By definition, we know that if $\mathbf u$ is a representative of $u\in\Gamma(\mathcal H^{p,q})$ and $\mathbf v^*$ is a  dual-representative of $v\in\Gamma(\mathcal H^{p,q})$ then
\begin{equation}
(u, v)=\pi_*\{\mathbf u,\mathbf v^*\},
\end{equation}
where $\{\cdot, \cdot\}$ is the canonical sesquilinear pairing (see page 268 in \cite{Demailly12}).  Now we have (see page 12 in \cite{Wang15})
\begin{eqnarray}
\frac{\partial}{\partial t^j} (u,v) &=&\frac{\partial}{\partial t^j} \pi_*\{\mathbf u,\mathbf v^*\} \\
&=&\pi_*( L_{V_j}\{\mathbf u, \mathbf v^*\}) \\
\label{eq:try} &=& \pi_*( \{L_j \mathbf u, \mathbf v^*\} + \{ \mathbf u, L_{\bar j}\mathbf v^*\}),
\end{eqnarray}
where $V_j$ is an arbitrary smooth $(1,0)$-vector field on $\mathcal X$ such that $\pi_*V_j=\partial/\partial t^j$ and 
\begin{equation}
L_j:=d^E\delta_{V_j}+\delta_{V_j} d^E, \ \ L_{\bar j} :=d^E \delta_{\overline{V_j}} +\delta_{\overline{V_j}} d^E.
\end{equation}
Here $d^E=\partial^E +\dbar$ denotes the Chern connection on $E$. In order to find a good expression of the Chern connection on $\mathcal H^{p,q}$, we shall introduce the following definition: 

\begin{definition}\label{de:horizontal-lift} Assume that $\mathcal X$ possesses a K\"ahler form $\omega$. A smooth $(1,0)$-vector field $V$  on $\mathcal X$ is said to be horizontal with respect to $\omega$ if 
\begin{equation}
i_t^*( \delta_{\overline V} \omega)=0,
\end{equation}
on $X_t$ for every $t\in B$. Moreover, for each $1\leq j\leq m$, we call $V_j$ the horizontal lift vector field of $\partial/\partial t^j$ with respect to $\omega$ if $V_j$ is horizontal with respect to $\omega$ and satisfies 
\begin{equation}
\pi_*(V_j)=\partial/\partial t^j.
\end{equation}
\end{definition}

Now we can prove that:

\begin{proposition}\label{pr:10-chern}  Assume that $\mathcal X$ possesses a K\"ahler form $\omega$ and each $V_j$ is the horizontal lift vector field of $\partial/\partial t^j$. Then 
\begin{equation}\label{eq:10-chern}
\pi_*\{ \mathbf u, L_{\bar j}\mathbf v^*\} =(u, \dbar_{t^j} v),
\end{equation}
for every smooth sections $u, \ v$ of $\mathcal H^{p,q}$.
\end{proposition}

\begin{proof} For bidegree reason, we have
 \begin{equation}
\pi_* \{ \mathbf u, L_{\bar j}\mathbf v^*\} =\pi_* \{ \mathbf u, [\dbar, \delta_{\overline{V_j}}]\mathbf v^*\},
\end{equation}
Thus it suffices to show that
 \begin{equation}\label{eq:10-chern-2}
(\dbar_{t^j} v)(t)= \mathbb H^t \left( (-1)^{p+q} *i_t^*[ \dbar, \delta_{\overline{V_j}}] \mathbf v^* \right).
\end{equation}
By Theorem \ref{th:holomorphic} and Definition \ref{de:dual-repre}, it suffices to check that
\begin{equation}
i_t^*[[\dbar, \delta_{\overline{V_j}}], \omega]=0.
\end{equation}
Since each $V_j$ is horizontal, the above equality is always true. The proof is complete.
\end{proof}

By Proposition \ref{pr:10-chern}, we have
\begin{equation}
\frac{\partial}{\partial t^j} (u,v)= \pi_* \{L_j \mathbf u, \mathbf v^*\}  +(u, \dbar_{t^j} v).
\end{equation}
By definition, the $(1,0)$-part of the Chern connection $D^{1,0}=\sum dt^j \otimes D_{t^j}$ should satisfy
\begin{equation}
\frac{\partial}{\partial t^j} (u,v)= (D_{t^j}u,v)  +(u, \dbar_{t^j} v).
\end{equation}
Thus we have:

\begin{proposition}\label{pr:chern} Assume that $\mathcal X$ possesses a K\"ahler form $\omega$ and each $V_j$ is the horizontal lift vector field of $\partial/\partial t^j$. Then the $(1,0)$-part of the Chern connection on $\mathcal H^{p,q}$ satisfies 
\begin{equation}\label{eq:uvjk3}
D_{t^j}u: t \mapsto \mathbb H^t\left(   i_t^*[\partial^E, \delta_{V_j}] \mathbf{u}\right),
\end{equation}
where $\mathbf u$ is an arbitrary representative of $u\in\Gamma(\mathcal H^{p,q})$.
\end{proposition}

\subsection{Curvature of $\mathcal H^{p,q}$} In this section, we shall assume that $\mathcal X$ possesses a K\"ahler form, say $\omega$, and $\mathcal H^{p,q}$ satisfies $\mathbf A$. For each $1\leq j\leq m$, we shall denote by $V_j$ the horizontal lift vector field of $\partial/\partial t^j$ with respect to $\omega$. 

Let $u,v$ be two holomorphic sections of $\mathcal H^{p,q}$. By Proposition \ref{pr:10-chern} and \eqref{eq:try}, we have
\begin{equation}
   \frac{\partial}{\partial t^j} (u,v)=\pi_* \{L_j \mathbf u, \mathbf v^*\}.
\end{equation}
Thus we have
\begin{eqnarray}
(u,v)_{j\bar k} &=& \frac{\partial}{\partial \bar t^k} \pi_* \{L_j \mathbf u, \mathbf v^*\}=\pi_* L_{\overline{V_k}}\{L_j \mathbf u, \mathbf v^*\}\\
&=& \pi_*( \{L_{\bar k}L_j \mathbf u, \mathbf v^*\} +\{L_j \mathbf u,L_k \mathbf v^*\}) \\
\label{eq:no-2}&=& \pi_*( \{[L_{\bar k},L_j] \mathbf u, \mathbf v^*\} + \{L_j L_{\bar k} \mathbf u, \mathbf v^*\}  + \{L_j \mathbf u,L_k \mathbf v^*\}) .
\end{eqnarray}
Since $u$ is a holomorphic section, by Theorem \ref{th:holomorphic}, for bidegree reason, we have
\begin{equation}\label{eq:re-u}
   \pi_*  \{L_{\bar k} \mathbf u, \mathbf v^*\} =0.
\end{equation}
Thus 
\begin{equation}\label{eq:no-1}
  0= \frac{\partial}{\partial t^j} \pi_*  \{L_{\bar k}\mathbf u, \mathbf v^*\}  = \pi_*( \{L_j L_{\bar k} \mathbf u, \mathbf v^*\} + \{L_{\bar k}\mathbf u, L_{\bar j}  \mathbf v^* \}) .
\end{equation}
By \eqref{eq:no-2} , we have
\begin{equation}\label{eq:uvjk}
 (u,v)_{j\bar k} = \pi_* \left( \{[L_{\bar k},L_j] \mathbf u, \mathbf v^*\} - \{L_{\bar k}\mathbf u, L_{\bar j}  \mathbf v^* \}   + \{L_j \mathbf u,L_k \mathbf v^*\}\right) . 
\end{equation}
For the last term, since each $V_j$ is horizontal, by \eqref{eq:star-lefschetz}, we have
\begin{eqnarray*}
\pi_*  \{L_j \mathbf u,L_k \mathbf v^*\} &= &\pi_*\left(  \{[\partial^E,\delta_{V_j}]\mathbf u,[\partial^E,\delta_{V_k}]\mathbf v^*\} +\{[\dbar,\delta_{V_j}]\mathbf u,[\dbar,\delta_{V_k}]\mathbf v^*\}\right) \\
&= & \left(i_t^*[\partial^E,\delta_{V_j}]\mathbf u, i_t^*[\partial^E,\delta_{V_k}]\mathbf v \right) - \left(\dbar V_j|_{X_t} \lrcorner~  u^t, \dbar V_k|_{X_t} \lrcorner ~ v^t\right).
\end{eqnarray*}
By the same reason,
we have
\begin{equation}
\pi_*  \{L_{\bar k} \mathbf u,L_{\bar j} \mathbf v^*\}= \left(i_t^*[\dbar ,\delta_{\overline{V_k}}]\mathbf u, i_t^*[\dbar,\delta_{\overline{V_j}}]\mathbf v \right) - \left(\partial \overline{V_k}|_{X_t} \lrcorner~  u^t, \partial \overline{V_j}|_{X_t} \lrcorner ~ v^t\right)
\end{equation}
By definition of the Chern connection, we have
\begin{equation}
    (\Theta_{j\bar k} u,v)=(D_{t^j}u ,D_{t^k} v)-(u,v)_{j\bar k}.
\end{equation}
Put 
\begin{equation}\label{eq:ab}
a_j^u= D_{t^j} u - i_t^* [\partial^E, \delta_{V_j}]\mathbf u ;  \  \ \ b_j^u=\dbar V_j|_{X_t} \lrcorner~  u^t,
\end{equation}
and
\begin{equation}
a_{\bar j}^v:=  i_t^*[\dbar, \delta_{\overline{V_j}}] \mathbf v, \ \ b_{\bar j}^v:= \partial \overline{V_j}|_{X_t}  \lrcorner~  v^t. 
\end{equation}
Then we have:

\begin{theorem}\label{th:theta_jk} Assume that $\mathcal X$ possesses a K\"ahler form $\omega$ and $\mathcal H^{p,q}$ satisfies $\mathbf A$. Let $u$ and $v$ be holomorphic sections of $\mathcal H^{p,q}$. Then we have
\begin{equation}\label{eq:theta_jk}
(\Theta_{j\bar k}u, v) =(b_j^u, b_k^v)-(a_j^u, a_k^v)+ \pi_* \{[L_j, L_{\bar k}] \mathbf u, \mathbf v^*\} + (a_{\bar k}^u, a_{\bar j}^v) -(b_{\bar k}^u, b_{\bar j}^v).
\end{equation}
\end{theorem}

\textbf{Remark}: For the middle term in the above formula, we shall use
\begin{equation}\label{eq:LJK-Lie}
 [L_j,L_{\bar k}]= [d^{E}, \delta_{[V_j,\bar V_k]}]+ \Theta(E,h)(V_j,\bar V_k).
\end{equation}
In order to study the other terms, we have to use H\"ormander's $L^2$-theory \cite{Hormander65} for the generalized $\dbar$-equation.

\subsection{Generalized $\dbar$-equation associated to the curvature formula}

We shall use the following lemma:

\begin{lemma}\label{le:re-u-1} If both $\mathbf u$ and $\mathbf u'$ are representatives of $u$ then 
\begin{equation}\label{eq:re-u-1}
  i_t^*L_{\bar k} (\mathbf u-\mathbf u')=0, \ \ i_t^*L_{j} (\mathbf u-\mathbf u')=0.
\end{equation}
\end{lemma}

\begin{proof} By definition, we have
\begin{equation}
\mathbf u- \mathbf{u}'=\sum dt^j\wedge a_j +\sum d\bar t^k \wedge b_k.
\end{equation}
Since
\begin{equation}\label{eq:re-u-2}
(d^E\delta_V +\delta_Vd^E)(df\wedge a)=d(Vf)\wedge a+ df\wedge (d^E\delta_V +\delta_Vd^E)a.
\end{equation}
Apply this formula to $f=t^j,\bar t^k$, $V=V_j, \overline{V_k}$ and $a=a_j, b_k$, we get \eqref{eq:re-u-1}.
\end{proof}

Now we can prove:

\begin{lemma}\label{le:abc-later} With the notation in Theorem \ref{th:theta_jk}, we have 
\begin{equation}\label{eq:abc-dbar-star}
(\dbar^t)^*a_j^u=(\dbar^t)^*b^v_{\bar j}=0, \ \ \forall \ t\in B, \ \forall\ 1\leq j\leq m.
\end{equation} 
Moreover, if $\mathcal H^{p,q}_t\subset \ker \partial^{E_t}$ for every $t\in B$ then 
\begin{equation}\label{eq:abc-later}
\dbar^t a_j^u=\partial^{E_t} b_i^u+c_j^u, \   \partial^{E_t} a_{\bar j}^v=-\dbar^t b_{\bar j}^v-c_{\bar j}^v, 
\end{equation}
and each $a_j^u$ is the $L^2$-minimal solution of  \eqref{eq:abc-later}, where
\begin{equation}
c_j^u:=(V_j \lrcorner~ \Theta(E,h))|_{X_t} \wedge u^t, \ c_{\bar j}^v:= (\overline{V_j} \lrcorner~ \Theta(E,h))|_{X_t} \wedge v^t.
\end{equation}
Assume further that $\mathcal H^{p,q}_t =\ker  \partial^{E_t} \cap \ker (\partial^{E_t})^*$ for every $t\in B$ then each $a_{\bar j}^v$ is also the $L^2$-minimal solution of \eqref{eq:abc-later}. 
\end{lemma}

\begin{proof} By \eqref{eq:ab}, we have
\begin{equation}
(\dbar^t)^*a_j^u=-  (\dbar^t)^*\left(i_t^* [\partial^E, \delta_{V_j}]\mathbf u\right).
\end{equation}
Since $(\dbar^t)^*=-*\partial^{E_t}*$, by \eqref{eq:star-lefschetz}, the  following equality
\begin{equation}\label{eq:abc1}
i_t^*\left(\partial^E [\partial^E, \delta_{V_j}]\mathbf u^*\right)=0.
\end{equation}
implies $(\dbar^t)^*a=0$. Notice that 
\begin{equation}
\partial^E [\partial^E, \delta_{V_j}]=[\partial^E, \delta_{V_j}]\partial^E,
\end{equation}
and
\begin{equation}
i_t^*\partial^E\mathbf u^*=\partial^{E_t}*u^t=0.
\end{equation}
Thus \eqref{eq:abc1} follows from \eqref{eq:re-u-2}. By the same proof, we have $(\dbar^t)^*b_{\bar j}^v=0$, thus \eqref{eq:abc-dbar-star} is true. Now let us prove \eqref{eq:abc-later}. By \eqref{eq:ab}, we have
\begin{equation}
\dbar^t a_j^u-\partial^{E_t} b_j^u=-i_t^*\left(
\dbar [\partial^E, \delta_{V_j}] \mathbf u + \partial^E[\dbar, \delta_{V_j}]\mathbf u\right).
\end{equation}
Since by our assumption, $\mathcal H_t^{p,q}\subset 
\ker \partial^{E_t}$, thus we have
\begin{equation}
i_t^*(\partial^E\mathbf u)=0, \ \ i_t^*(\dbar\mathbf u)=0,
\end{equation}
by \eqref{eq:re-u-2}, we have
\begin{equation}
i_t^* [\partial^E, \delta_{V_j}] \dbar\mathbf u = i_t^* [\dbar, \delta_{V_j}]\partial^E\mathbf u=0.
\end{equation}
Thus
\begin{equation}
\dbar^t a_j^u-\partial^{E_t} b_j^u=-i_t^*\left(
[\dbar, [\partial^E, \delta_{V_j}] ]\mathbf u + [\partial^E, [\dbar, \delta_{V_j}]]\mathbf u\right)= i_t^*[ \delta_{V_j}, [\dbar, \partial^E]]\mathbf u=c_j^u.
\end{equation}
By the same proof, we have
\begin{equation*}
\partial^{E_t} a_{\bar j}^v=-\dbar^t b_{\bar j}^v-c_{\bar j}^v.
\end{equation*} 
Thus \eqref{eq:abc-later} is true. Now let us prove the last part. Since $v$ is a holomorphic section, by Theorem \ref{th:holomorphic}, we have that each $a_{\bar j}^v$ has no $\dbar^t$-harmonic part. By our assumption, each $\dbar^t$-harmonic space is equal to the  $\partial^{E_t}$-harmonic space, we know that each $a_{\bar j}^v$ has no $\partial^{E_t}$-harmonic part. Thus it is sufficient to prove that each $a_{\bar j}^v$ is $(\partial^{E^t})^*$-closed.  Since $(\partial^{E^t})^*=-*\dbar^t*$, it is sufficient to show that
\begin{equation}
\dbar^t*a_{\bar j}^v=0.
\end{equation}
By \eqref{eq:star-lefschetz}, we know that
\begin{equation}\label{eq:abc-v8}
\dbar^t*a_{\bar j}^v=\dbar^t i_t^*[\dbar, \delta_{\overline{V_j}}] \mathbf v^*=i_t^*[\dbar, \delta_{\overline{V_j}}]\dbar \mathbf v^*.
\end{equation}
Moreover, by our assumption, each $v^t$ is also in the $\partial^{E_t}$-harmonic space. Thue we have
\begin{equation}
i_t^*\dbar \mathbf v^*=\dbar^t*v^t=0,
\end{equation}
by \eqref{eq:re-u-2} and \eqref{eq:abc-v8}, we know that $\dbar^t*a_{\bar j}^v=0$. The proof is complete.
\end{proof}

By \eqref{eq:theta_jk} and the above lemma, the $L^2$-estimates for the generalized $\dbar$-equation (see \eqref{eq:abc-later}) determine the positivity of $\mathcal H^{p,q}$. We shall show how to use the following version of H\"ormander's $L^2$-theory \cite{Hormander65} to study the curvature of $\mathcal H^{p,q}$:

\begin{theorem}\label{th:L2} Let $(E, h)$ be a Hermitian vector bundle over an $n$-dimensional compact complex manifold $X$. Let $q\geq 0$ be an integer. Assume that $X$ possesses a Hermitian metric $\omega$ such that $\omega_{\max\{q,1\}}$ is $\dbar$-closed. Let $v$ be a smooth $\dbar$-closed $E$-valued $(n,q+1)$-form. Assume that
\begin{equation*}
    i\Theta(E,h)\wedge\omega_q >0 \ \text{on} \ X,\ (resp. \ i\Theta(E,h)\wedge\omega_q \equiv 0 \ \text{on} \ X ),
\end{equation*}
and
\begin{equation*}
    I(v):= \inf_{v=\partial^E b+c} ||b||^2_{\omega} + ([i\Theta(E,h), \Lambda_{\omega}]^{-1}c,c)_{\omega} <\infty, \ (resp. \ I(v):= \inf_{v=\partial^E b} ||b||^2_{\omega}<\infty ) .
\end{equation*}
Then there exists a smooth $E$-valued $(n,q)$-form $a$ such that $\dbar a=v$ and
\begin{equation}\label{eq:L2q}
    ||a||_{\omega}^2 \leq I(v).
\end{equation}
\end{theorem}

\begin{proof} Let $\gamma$ be an arbitrary $E$-valued smooth $(n-q-1,0)$-form on $X$. Put
\begin{equation}
    T=i^{(n-q-1)^2}\{\gamma, \gamma\}\wedge \omega_{q}, \ \ u=\gamma\wedge\omega_{q+1}.
\end{equation}
Since $\dbar\omega_q=0$, one may check that $ i\partial\dbar T$ can be written as
\begin{equation}\label{eq:T}
    -2{\rm Re} \langle \dbar\dbar^* u,u\rangle\omega_{n} + |\dbar^* u|^2\omega_{n}
    +i^{(n-q-1)^2}\{i\Theta(E,h)\gamma, \gamma\} \wedge\omega_{q}-S,
\end{equation}
where
\begin{equation}
    S= i^{(n-q)^2}\{\dbar\gamma, \dbar\gamma\}\wedge\omega_{q}.
\end{equation}
By Lemma 4.2 in Berndtsson's lecture notes \cite{Bern10}, we have
\begin{equation}
S=(|\omega_{q+1}\wedge\dbar\gamma|^2-|\dbar\gamma|^2) \omega_n.
\end{equation}
Since $\omega_{\max\{q,1\}}$ is $\dbar$-closed, we have $\dbar\omega_{q+1}=0$, thus 
\begin{equation}
S=(|\dbar u|^2-|(\partial^E)^*u|^2) \omega_n.
\end{equation}
Notice that $\int_X i\partial\dbar T =0$, thus we have
\begin{eqnarray}\label{eq:T11}
||\dbar u||^2+||\dbar^* u||^2-||(\partial^E)^*u||^2 &=& \int_Xi^{(n-q-1)^2}\{i\Theta(E,h)\gamma, \gamma\} \wedge\omega_{q}\\
\label{eq:T12} &=& ([i\Theta(E,h), \Lambda_{\omega}]u,u),
\end{eqnarray}
where $\Lambda_{\omega}$ is the adjoint of $\omega\wedge\cdot$. By H\"ormander's $L^2$ theory, it suffices to show that
\begin{equation}\label{eq:T2}
    |(v,u)|^2\leq I(v)(||\dbar u||^2+||\dbar^* u||^2), \ \forall \ \text{smooth} \ u,
\end{equation}
which follows from \eqref{eq:T11} and \eqref{eq:T12} by the Cauchy-Schwartz inequality. 
\end{proof}

We shall also use the following generalized version of Theorem \ref{th:L2} (see Remark 13.5 in \cite{Demailly96} for related results):

\begin{theorem}\label{th:L2-1} Let $(E, h)$ be a Hermitian vector bundle over an $n$-dimensional compact complex manifold $X$. Let $q\geq 0$ be an integer. Assume that $X$ possesses a Hermitian metric $\omega$ such that $\omega_{\max\{q,1\}}$ is $\dbar$-closed. Let $v$ be a smooth $\dbar$-closed $E$-valued $(n,q+1)$-form. Assume that $i\Theta(E,h)\wedge\omega_q \geq 0$ on $X$ and
and
\begin{equation*}
    I(v):= \inf_{v=\partial^E b+c} \left\{ ||b||^2_{\omega} + \lim_{\varepsilon\to 0} \left( ([i\Theta(E,h) , \Lambda_{\omega}]+\varepsilon)^{-1}c,c\right)_{\omega} \right\} <\infty.
\end{equation*}
Then there exists a smooth $E$-valued $(n,q)$-form $a$ such that $\dbar a=v$ and
\begin{equation}\label{eq:L2q-1}
    ||a||_{\omega}^2 \leq I(v).
\end{equation}
\end{theorem}

\begin{proof} By \eqref{eq:T11} and \eqref{eq:T12}, for every $\varepsilon>0$, we have 
\begin{equation}\label{eq:T2-1}
    |(v,u)|^2\leq I(v)(||\dbar u||^2+||\dbar^* u||^2+\varepsilon ||u||^2), \ \forall \ \text{smooth} \ u,
\end{equation}
By  H\"ormander's $L^2$ theory, there exist $a_{\varepsilon}$ and $\alpha_{\varepsilon}$ such that
\begin{equation}
\dbar a_{\varepsilon}+\sqrt{\varepsilon} \alpha_{\varepsilon} =v,
\end{equation}
and
\begin{equation}
||a_\varepsilon||^2+||\alpha_{\varepsilon}||^2\leq I(v).
\end{equation}
Thus $\sqrt{\varepsilon} \alpha_{\varepsilon}\to 0$ in the sense of current as $\varepsilon\to 0$. By taking a weak limit, say $a'$, of $a_{\varepsilon}$, we know that there exists $a'$ such that $\dbar a'=v$ in the sense of current and
\begin{equation}
||a'||^2 \leq I(v).
\end{equation}
Let $a$ be the $L^2$-minimal solution. Then $a$ fits our needs.
\end{proof}

\subsection{Curvature of $\mathcal H^{n,q}$} Theorem \ref{th:L2-1} implies that there is a good $L^2$-estimate for \eqref{eq:abc-later} if $p=n$. We shall show how to use it to prove Theorem \ref{th:main} in this subsection. 

Let us study the middle term in \eqref{eq:theta_jk} first. By \eqref{eq:LJK-Lie}, we have
\begin{equation}\label{eq:Ljk}
\pi_* \{[L_j, L_{\bar k}] \mathbf u_j, \mathbf u_k^*\}=\left (\Theta(E,h)(V_j,\bar V_k) u_j, u_k\right ) + (\partial^{E_t}\delta_{[V_j,\bar V_k]}u_j, u_k).
\end{equation}
The first term is clear. For the second term, we need the following lemma: 

\begin{lemma}\label{le:key} $\sum(\partial^{E_t}\delta_{[V_j,\bar V_k]}u_j, u_k)$ can be written as $A + Nak$, where
\begin{equation}\label{eq:A}
A:=\sum i^{(n-q)^2} \pi_*\{c_{j\bar k}(\omega)\omega_{q-1} \wedge i\Theta(E,h) \mathbf u_j^*, \mathbf u_k^*\}, \ c_{j\bar k}(\omega):=\langle V_j, V_k \rangle_{\omega} ,
\end{equation}
and $Nak:=\sum  (c_{j\bar k}(\omega) \dbar^t *u_j^t, \dbar^t *u_k^t)\geq 0$.
\end{lemma}

We shall prove the above Lemma later. By the above lemma, we have
\begin{equation}
\sum \pi_* \{[L_j, L_{\bar k}] \mathbf u_j, \mathbf u_k^*\} =A+B+ Nak,
\end{equation}
where
\begin{equation}\label{eq:B}
B:=\sum\left (\Theta(E,h)(V_j,\bar V_k)u_j, u_k\right ).
\end{equation}
Thus by \eqref{eq:theta_jk}, put
\begin{equation}
b=\sum b_j^{u_j}, \ \ a=\sum a_j^{u_j},
\end{equation}
then we have (notice that  $b_{\bar j}^{u_k}=0$ since $p=n$)
\begin{equation}\label{eq:theta_jk1}
\sum (\Theta_{j\bar k}u_j, u_k) =||b||^2-||a||^2+A+B+Nak +Gri,
\end{equation}
where
\begin{equation}\label{eq:Gri}
Gri:=\sum (a_{\bar k}^{u_j}, a_{\bar j}^{u_k}),
\end{equation}
is non-negative if $u_j=\xi_j u$. Put
\begin{equation}
C_{\varepsilon}:= \left(( [i\Theta(E^t,h), \Lambda_{\omega^t}]+\varepsilon) ^{-1} c, c\right), \ \ c:=\sum c_j^{u_j}.
\end{equation}
Then we have
\begin{equation}
\sum (\Theta_{j\bar k}u_j, u_k) = A+B-C_{\varepsilon}+Nak+Gri+R_{\varepsilon},
\end{equation}
where
\begin{equation}
R_{\varepsilon}:= \left(( [i\Theta(E^t,h), \Lambda_{\omega^t}]+\varepsilon) ^{-1} c, c\right) + ||b||^2-||a||^2.
\end{equation}

\medskip

\textbf{Proof of Theorem \ref{th:main}} By \eqref{eq:abc-later} and Theorem \ref{th:L2-1}, we know that $R_\varepsilon$ is always non-negative. In order to show that $\mathcal H^{n,q}$ is Griffiths semi-positive, it suffices to prove $$\lim_{\varepsilon\to 0} \ (A+B-C_\varepsilon)\geq 0, $$
which follows from the following lemma:

\begin{lemma}\label{le: ABC} If $\Theta(E,h)$ is $q$-semipositive with respect to $\omega$ then 
\begin{equation}
A+B+\varepsilon\sum (c_{j\bar k}(\omega) u_j, u_k)-C_{\varepsilon}\geq 0, \ \ \forall \ \varepsilon >0.
\end{equation}
\end{lemma}

\begin{proof} Put
\begin{equation}
I_\varepsilon:=A+B+\varepsilon\sum (c_{j\bar k}(\omega) u_j, u_k)-C_{\varepsilon},
\end{equation}
and
\begin{equation}
T_\varepsilon:=\omega_q\wedge i \Theta(E,h)
+\varepsilon\omega_{q+1}\otimes Id_{E}, \ T_\varepsilon^t:=T_\varepsilon|_{X_t}.
\end{equation}
We claim that for each $j$ one may choose a dual representative $\mathbf u_j^*$ of $u_j$ such that
\begin{equation}\label{eq:ABCR}
i_t^*\delta_{V_j}  \left( T_\varepsilon \wedge \mathbf u_j^*\right)\equiv 0,
\end{equation}
on $X_t$ for every $t\in B$ and
\begin{equation}\label{eq:ABC-F}
I_\varepsilon \wedge i^{m^2}dt\wedge \overline{dt}=\pi_*\left( c_q \{T_\varepsilon \wedge \mathbf u^*, \mathbf u^*\}\right) \geq 0, \ \mathbf u^*:= \sum \mathbf u_j^*\wedge\delta_{V_j}dt,
\end{equation}
where $c_q:= i^{(m+n-q-1)^2}$ and $dt:=dt^1\wedge \cdots \wedge dt^m$.

In fact, \eqref{eq:ABCR} is equivalent to 
\begin{equation}\label{eq:ABC?}
\omega^t_q \wedge (V_j \lrcorner~ i \Theta(E,h))|_{X_t} \wedge *u^t_j + T_\varepsilon^t \wedge i_t^* \delta_{V_j} \mathbf u_j^* \equiv 0.
\end{equation}
Notice that
\begin{equation}
Q_\varepsilon^{-1}(T_\varepsilon^t  \wedge i_t^* \delta_{V_j} \mathbf u_j^*)= \omega^t_{q+1} \wedge  i_t^* \delta_{V_j} \mathbf u_j^*, \  Q_\varepsilon:=[\Theta(E^t,h), \Lambda_{\omega^t}]+\varepsilon,
\end{equation}
Thus there exists $\mathbf u_j^*$ such that \eqref{eq:ABCR} is true. Now choose $\mathbf u_j^*$ that  satisfies \eqref{eq:ABCR}, then
\begin{equation}
c=\sum i^{(n-q)^2+1} T_\varepsilon^t \wedge i_t^* \delta_{V_j} \mathbf u_j^*,
\end{equation}
which implies
\begin{equation}
Q_\varepsilon^{-1}c=\sum i^{(n-q)^2+1} \omega^t_{q+1} \wedge i_t^* \delta_{V_j} \mathbf u_j^*.
\end{equation}
Thus we have
\begin{equation}
C_\varepsilon=(Q_\varepsilon^{-1}c,c)=\pi_*\left( i^{(n-q-1)^2} \{T_\varepsilon \wedge \sum \delta_{V_j}\mathbf u_j^*, \sum \delta_{V_j} \mathbf u_j^*\}\right).
\end{equation}
Recall that
\begin{equation}
A=\sum  i^{(n-q)^2} \pi_* \{c_{j\bar k}(\omega)\omega_{q-1} \wedge i\Theta(E,h) \mathbf u_j^*, \mathbf u_k^*\},
\end{equation}
and 
\begin{equation}
B=\sum  i^{(n-q)^2} \pi_* \{ \omega_q \wedge \Theta(E,h)(V_j,\bar V_k)  \mathbf u_j^*, \mathbf u_k^*\} .
\end{equation}
Notice that
\begin{equation}\label{eq:ABCIjk}
i_t^*\left(i\delta_{V_j}\delta_{\overline V_k} T_\varepsilon \right)= \omega^t_q \wedge \Theta(E,h)(V_j,\bar V_k)  + i\Theta(E_t,h)\wedge c_{j\bar k}(\omega)\omega^t_{q-1} +\varepsilon c_{j\bar k}(\omega)\omega^t_{q}.
\end{equation}
Thus
\begin{equation}
A+B+\varepsilon\sum (c_{j\bar k}(\omega) u_j, u_k)=\sum i^{(n-q)^2} \pi_*\left\{ (i\delta_{V_j}\delta_{\overline V_k} T_\varepsilon) \wedge \mathbf u_j^*, \mathbf u_k^* \right\}.
\end{equation}
Notice that
\begin{eqnarray*}
\delta_{V_j}\delta_{\overline V_k} \{T_\varepsilon\wedge \mathbf u_j^*, \mathbf u_k^*\} & = &  (-1)^{n-q} \{(\delta_{V_j}T_\varepsilon)\wedge \mathbf u_j^*, \delta_{V_k} \mathbf u_k^* \}  - \{(\delta_{\overline V_k} T_\varepsilon) \wedge  \delta_{V_j} \mathbf u_j^*,  \mathbf u_k^* \} \\
&  & + \ \{ (\delta_{V_j}\delta_{\overline V_k} T_\varepsilon) \wedge \mathbf u_j^*, \mathbf u_k^* \} + (-1)^{n-q} \{T_\varepsilon \wedge \delta_{V_j} \mathbf u_j^*,\delta_{V_k} \mathbf u_k^* \}.
\end{eqnarray*}
By  \eqref{eq:ABCR}, we have
\begin{equation}
\pi_*\{(\delta_{V_j}T_\varepsilon)\wedge \mathbf u_j^*, \delta_{V_k} \mathbf u_k^* \} =- \pi_*  \{T_\varepsilon \wedge \delta_{V_j} \mathbf u_j^*,\delta_{V_k} \mathbf u_k^* \},
\end{equation}
and
\begin{equation}
\pi_* \{(\delta_{\overline V_k} T_\varepsilon) \wedge  \delta_{V_j} \mathbf u_j^*,  \mathbf u_k^* \} = (-1)^{n-q} \pi_* \{T_\varepsilon \wedge \delta_{V_j} \mathbf u_j^*,\delta_{V_k} \mathbf u_k^* \} .
\end{equation}
Thus
\begin{equation}
\pi_* \delta_{V_j}\delta_{\overline V_k} \{T_\varepsilon\wedge \mathbf u_j^*, \mathbf u_k^*\} = \pi_* \{ (\delta_{V_j}\delta_{\overline V_k} T_\varepsilon) \wedge \mathbf u_j^*, \mathbf u_k^* \}
- (-1)^{n-q} \pi_*  \{T_\varepsilon \wedge \delta_{V_j} \mathbf u_j^*,\delta_{V_k} \mathbf u_k^* \},
\end{equation}
which implies that
\begin{equation}
\sum i^{(n-q)^2}  \pi_*\left( i\delta_{V_j}\delta_{\overline V_k} \{T_\varepsilon\wedge \mathbf u_j^*, \mathbf u_k^*\}\right) =I_\varepsilon.
\end{equation}
Notice that for bi-degree reason, we have
\begin{equation}
(\delta_{V_j}\delta_{\overline V_k} \{T_\varepsilon\wedge \mathbf u_j^*, \mathbf u_k^*\})\wedge dt\wedge\overline{dt} = (-1)^m \{T_\varepsilon\wedge \mathbf u_j^*, \mathbf u_k^*\} \wedge \delta_{V_j} dt \wedge \overline{\delta_{V_k} dt }.
\end{equation}
Thus  \eqref{eq:ABC-F} is true. The proof is complete.
\end{proof}

\textbf{Proof of the remark behind Theorem \ref{th:main}}: Since $p=n$, we know that 
\begin{equation}
b_{\bar j}^{u_k}=c_{\bar j}^{u_k}=0.
\end{equation}
Thus by Lemma \ref{le:abc-later}, if $\mathcal H^{p,q}_t =\ker  \partial^{E_t} \cap \ker (\partial^{E_t})^*$ for every $t\in B$ then
we have
\begin{equation}
a_{\bar j}^{u_k}=0,
\end{equation}
which implies that $Gri=0$. Thus $\mathcal H^{n,q}$ is Nakano semi-positive.

\medskip

\textbf{Proof of Lemma \ref{le:key}}: Put
\begin{equation}
I_{jk}:=(\partial^{E_t}\delta_{[V_j,\bar V_k]}u_j, u_k).
\end{equation}
Denote by $V$ the $(1,0)$-part of $[V_j,\bar V_k]$, then we have 
\begin{equation}
I_{jk}=\int_{X_t} \{\partial^{E_t}\delta_V u_j^t, *u_k^t\}=(-1)^{n-q}\int_{X_t}\{\delta_V u_j^t, \dbar^t*u_k^t\}.
\end{equation}
Since
\begin{equation}
u_j^t=i^{(n-q)^2}\omega^t_q\wedge*u^t_j,
\end{equation}
and
\begin{equation}\label{eq:key-123}
i^{(n-q)^2} \omega^t_q\wedge\dbar^t*u_k^t=\dbar^t u_k^t=0,
\end{equation}
we have
\begin{equation}
I_{jk}=(-1)^{n-q}i^{(n-q)^2} \int_{X_t}\{(\delta_V  \omega^t_q)\wedge * u_j^t, \dbar^t*u_k^t\}.
\end{equation}
By definition, $\delta_V  \omega^t_q$ is the $(q-1,q)$-part of
\begin{equation}
i_t^*\left((L_{V_j}\bar V_{k}) ~\lrcorner~ \omega_q\right).
\end{equation}
Since
\begin{equation}
(L_{V_j}\bar V_{k}) ~\lrcorner~ \omega_q=L_{V_j} ( \bar V_{k} ~\lrcorner~ \omega_q)- \bar V_{k} ~\lrcorner~ L_{V_j}\omega_q,
\end{equation}
and 
\begin{equation}
\bar V_{k} ~\lrcorner~ \omega_q=  -i\sum c_{j\bar k} (\omega) dt^j\wedge \omega_{q-1}, \ \     L_{V_j}\omega_q=   d\left(i\sum c_{j\bar k} (\omega) d\bar t^k\wedge \omega_{q-1}\right).
\end{equation}
Thus 
\begin{equation}
i_t^*\left((L_{V_j}\bar V_{k}) ~\lrcorner~ \omega_q\right)=i_t^* d\left(i c_{j\bar k} (\omega) \wedge \omega_{q-1}\right),
\end{equation}
and 
\begin{equation}
\delta_V  \omega^t_q=\dbar^t\left(i c_{j\bar k} (\omega) \wedge \omega^t_{q-1}\right).
\end{equation}
Now $I_{jk}$ can be written as 
\begin{equation*}
i^{(n-q+1)^2}\int_{X_t}\{\dbar^t( c_{j\bar k} (\omega) \wedge \omega^t_{q-1} \wedge * u_j^t), \dbar^t*u_k^t\} -\{c_{j\bar k} (\omega) \wedge \omega^t_{q-1}\wedge ~ \dbar^t *u_j^t, \dbar^t*u_k^t\}.
\end{equation*}
By \eqref{eq:key-123}, each $\dbar^t*u_k^t$ is primitive, thus
\begin{equation}
-\sum i^{(n-q+1)^2}\int_{X_t} \{c_{j\bar k} (\omega) \wedge \omega^t_{q-1}\wedge ~ \dbar^t *u_j^t, \dbar^t*u_k^t\}=Nak.
\end{equation}
Now it suffices to show 
\begin{equation}\label{eq:key-1234}
A=\sum  i^{(n-q+1)^2} \int_{X_t}\{\dbar^t( c_{j\bar k} (\omega) \wedge \omega^t_{q-1} \wedge * u_j^t), \dbar^t*u_k^t\}.
\end{equation}
Notice that the right hand side can be written as
\begin{equation}\label{eq:key-12345}
\sum i^{(n-q)^2}(-i) \int_{X_t}\{ c_{j\bar k} (\omega) \wedge \omega^t_{q-1} \wedge * u_j^t, \partial^{E_t}\dbar^t*u_k^t\}.
\end{equation}
Since each $u_k^t$ is harmonic, we have $\partial^{E_t}*u_k^t\equiv 0$ and
\begin{equation}\label{eq:key-12}
\partial^{E_t}\dbar^t*u_k^t=\Theta(E_t,h) *u_k^t.
\end{equation}
Thus \eqref{eq:key-1234} follows from \eqref{eq:key-12345} and \eqref{eq:key-12}. The proof is complete.

\section{Curvature of the Weil-Petersson metric}

We shall prove Theorem \ref{th:wp-general} in this section. Let $\pi: \mathcal X \to B$ be a proper holomorphic submersion. Then we have the classical Kodaira-Spencer map
\begin{equation}
\kappa :v \mapsto \kappa(v) \in H^{0,1}(T_{X_t})\simeq H^{n,n-1}(T^*_{X_t})^*, \ \forall  \ v\in T_tB, \ t\in B.
\end{equation} 
we shall prove that:

\begin{proposition}\label{pr:ks-dual} Assume that the dimension of $H^{0,1}(T_{X_t})$ does not depend on $t$. Then \begin{equation}
\kappa: v\mapsto \kappa(v)\in H^{n,n-1}(T^*_{X_t})^*,  \ \forall  \ v\in T_tB,
\end{equation} 
defines a holomorphic bundle map from $T_B$ to the dual bundle of
\begin{equation}
\mathcal H^{n,n-1}:=\{H^{n,n-1}(T^*_{X_t})\}_{t\in B}.
\end{equation} 
\end{proposition}

\begin{proof} Let $v: t \mapsto v^t$ be a holomorphic vector field on $B$. Let $V$ be a smooth $(1,0)$-vector field on the total space such that $\pi_* V=v$. Then we know that for each $t$, $(\dbar V)|_{X_t}$ defines a representative of $\kappa(v^t)$. It suffices to prove that
\begin{equation}
\pi_*(\mathbf u \wedge \dbar V)
\end{equation}
is holomorphic for every holomorphic section $u$ of $\mathcal H^{n,n-1}$, where $\mathbf u$ is an arbitrary representative of $u$. Let us write
\begin{equation}
\dbar \mathbf u=\sum dt^j\wedge a_j +\sum d\bar t^k \wedge b_k.
\end{equation} 
By the proof of Theorem \ref{th:holomorphic}, we know that each $b_k|_{X_t}$ is $\dbar^t$-closed and the cohomology class of $b_k|_{X_t}$ does not depend of the choice of $b_k$. Moreover, if $u$ is a holomorphic section then the cohomology class of $b_k|_{X_t}$ is zero. Thus 
\begin{equation}
\pi_*(b_k \wedge \dbar V)=0,
\end{equation}
and
\begin{equation}
\dbar \pi_*(\mathbf u \wedge \dbar V) = \pi_*(\dbar \mathbf u \wedge \dbar V)=\sum d\bar t^k \wedge \pi_*(b_k \wedge \dbar V)=0,
\end{equation}
which implies that $ \pi_*(\mathbf u \wedge \dbar V)$ is holomorphic. The proof is complete.
\end{proof}

\medskip

\textbf{Proof of Theorem \ref{th:wp-general}}:  Let $v: t \mapsto v^t$ be an arbitrary holomorphic vector field on $B$. Let us denote by $||v||_{WP}$ the upper semicontinuous regularization of
\begin{equation}
t\mapsto ||v^t||_{WP}.
\end{equation}
Then it suffices to prove that $\log ||v||_{WP}$ is plurisubharmonic or equal to $-\infty$ identically. Since there is a Zariski open subset, say $B_0$, of $B$ such that  the dimension of $H^{n,n-1}(T^*_{X_t})$ is a constant on $B_0$, by Theorem \ref{th:main}, we know that $\mathcal H^{n,n-1}$ is Griffiths-semipositive on the complement of a proper analytic subset. Thus by Proposition \ref{pr:ks-dual}, we know that $\log ||v||_{WP}$ is plurisubharmonic or equal to $-\infty$ identically on the complement of a proper analytic subset. Now it is sufficient to prove that $||v||_{WP}$ is locally bounded from above. 

Let $V$ be a smooth $(1,0)$-vector field on the total space such that $\pi_* V=v$.  By definition, we know that
\begin{equation}
||v^t||_{WP} \leq ||(\dbar V)|_{X_t}||_{H^{0,1}(T_{X_t})}.
\end{equation}
Since
\begin{equation}
t\mapsto ||(\dbar V)|_{X_t}||_{H^{0,1}(T_{X_t})},
\end{equation}
is smooth, we know that $||v||_{WP}$ is locally bounded from above. The proof is complete.

\section{Acknowledgement} 

I would like to thank Bo Berndtsson for many inspiring discussions (in particular for his suggestion to relate our results to the Weil-Petersson geometry), and for his many useful comments on this paper. Thanks are also given to Bo-Yong Chen and Qing-Chun Ji for their constant support and encouragement.

\end{document}